\documentclass[runningheads]{llncs}
\usepackage{graphicx}
\usepackage{amsmath}
\usepackage{amssymb}
\usepackage{amsthm}
\usepackage{mathrsfs} 
\newcommand{\Ob}{\operatorname{Ob}}
\newcommand{\Hom}{\operatorname{Hom}}

\newcommand{\id}{\operatorname{id}}
\newcommand{\cat}[1]{\mathbf{#1}}
\newcommand{\norm}[1]{\|{#1}\|}
\newcommand{\V}[1]{\vec{#1}}

\begin{document}
\title{Categorical magnitude and entropy\thanks{SC acknowledges the support of Marcella Bonsall through her SURF fellowship. }}  
%
%
\author{Stephanie Chen\inst{1}
\and
Juan Pablo Vigneaux\inst{2}
}

\authorrunning{S. Chen and J. P. Vigneaux}
%
\institute{California Institute of Technology, Pasadena CA 91125, USA\\
\email{schen7@caltech.edu} \and
Department of Mathematics, California Institute of Technology, \\ Pasadena CA 91125, USA\\
\email{vigneaux@caltech.edu}}
\maketitle              
\begin{abstract}
Given any finite set equipped with a probability measure, one may compute its Shannon entropy or information content. The entropy becomes the logarithm of the cardinality of the set when the uniform probability is used.  Leinster introduced a notion  of Euler characteristic for certain finite categories, also known as magnitude, that can be seen as a categorical generalization of cardinality. This paper aims to connect the two ideas by considering the extension of Shannon entropy to finite categories endowed with probability, in such a way that the magnitude is recovered when a certain choice of ``uniform'' probability is made. 

\end{abstract}

This version of the article has been accepted for publication, after peer review (when applicable) and is subject to Springer Nature’s AM terms of use but is not the Version of Record and does not reflect post-acceptance improvements, or any corrections. The Version of Record is available online at: https://doi.org/10.1007/978-3-031-38271-0\_28. 

\section{Introduction}
Given a finite set $X$ endowed with a probability measure $p$, its Shannon entropy \cite{Shannon} is given by
\begin{equation}\label{eq:entropy}
H(p) = -\sum_{x\in X} p(x)\ln p(x). 
\end{equation}
In particular, taking the uniform probability $u:x\mapsto 1/|X|$ yields $H(u) = \ln|X|$.  We may thus view Shannon entropy as a \emph{probabilistic} generalization of cardinality. A \emph{categorical} generalization of cardinality may be found in the Euler characteristic or magnitude of finite ordinary categories \cite{Leinster}, defined as follows. 

Let $\cat{A}$ be a finite category. The zeta function $\zeta: \Ob(\cat{A}) \times \Ob(\cat{A}) \to \mathbb{Q}$ is given by $\zeta(x,y) = |\text{Hom}(x,y)|$, the cardinality of the hom-set, for any $x,y\in \Ob(\cat{A})$. A weighting on $\cat{A}$ is a function $k^{\bullet}: \Ob(\cat{A}) \to \mathbb{Q}$ such that 
\begin{equation}
\label{eq: weighting}
\sum_{b \in \Ob(\cat{A})} \zeta(a,b)k^b = 1
\end{equation}
for all $a \in \Ob(\cat{A})$. Similarly, a coweighting on $\cat{A}$ is a function $k_{\bullet}: \Ob(\cat{A}) \to \mathbb{Q}$ such that 
\begin{equation}
\label{eq: coweighting}
\sum_{b \in \Ob(\cat{A})} \zeta(b,a)k_b = 1
\end{equation}
for all $a \in \Ob(\cat{A})$. Equivalently, one may view a coweighting on $\cat{A}$ as a weighting on $\cat{A}^{\text{op}}$. If $\cat{A}$ admits both a weighting and a coweighting, then
\begin{equation}\sum_{a\in\Ob(\cat{A})} k^a = \sum_{a\in\Ob(\cat{A})} k_a;
\end{equation}
in this case, the \emph{magnitude} of $\cat A$,   denoted $\chi(\cat A)$, is defined as the common value of both sums. 

Magnitude enjoys  algebraic properties reminiscent of cardinality, such as $\chi(\cat A\coprod \cat B) = \chi(\cat A)+\chi(\cat B)$ and $\chi(\cat A \times \cat B) =\chi(\cat A)\chi(\cat B)$. Moreover, when $\cat A$ is a discrete category (i.e. it only has identity arrows),  $\chi(\cat A) = |\Ob (\cat A)|$.  Hence magnitude may be regarded as a categorical generalization of cardinality.

We ask if there is an extension of Shannon entropy from finite sets to finite categories that gives a probabilistic generalization of the magnitude; in particular, we want this extension to give us the logarithm of the magnitude under some ``uniform'' choice of probabilities and to coincide with Shannon entropy when specialized to discrete categories.

The rest of this paper is organized as follows. In Section 2, we introduce the category $\cat{ProbFinCat}$ whose objects are categorical probabilistic triples $(\cat{A},p,\phi)$ and whose morphisms are probability-preserving functors. In Section \ref{sec:entropy}, we define a function $\mathcal{H}$ of categorical probabilistic triples that shares analogous properties to those used by Shannon \cite{Shannon} to characterize the entropy \eqref{eq:entropy}. This function $\mathcal H$ allows us to recover the set-theoretical Shannon entropy and the categorical magnitude for particular choices of $p$ or $\phi$.  In Section 4,  we discuss  the possibility of characterizing the ``information loss'' given by $\mathcal H$ in the spirit of \cite{Entropy_Characterization}.

\section{Probabilistic categories}\label{sec:prob_cat}

\begin{definition}\label{def:cat_prob_triples}
    A \emph{categorical probabilistic triple}  $(\cat{A},p,\phi)$ consists of
\begin{enumerate}
    \item   a finite category $\cat A$,
    \item a probability $p$ on $\Ob (\cat A)$, and
    \item  a function $\phi:\Ob (\cat A)\times \Ob (\cat A)\to[0,\infty)$ 
 such that $\phi(a,a) > 0$ for all objects $a$ of $\cat{A}$, and $\phi(b,b')=0$ whenever there is no arrow from $b$ to $b'$ in $\cat A$.
\end{enumerate}
\end{definition}

The definition gives a lot of flexibility for $\phi$, provided it reflects the incidence relations in the category. It might be the $\zeta$ function introduced above. Alternatively, it might be a measure of similarity between two objects, see next section. Finally, it might be a transition kernel, in which case for every $a\in \Ob (\cat A)$, the function $\phi(\cdot,a)$ is a probability mass function on the objects $b$ such that an arrow $a\to b$ exists in $\cat A$; we treat this case in more detail in Section \ref{sec:special_cases}.

\begin{remark}
    Given a categorical probabilistic triple $(\cat{A}, p, \phi)$, 
set $N=|\Ob (\cat A)|$ and enumerate the objects of $\cat{A}$, in order to  introduce a matrix $Z_{\phi}$ of size $N\times N$ whose $(i,j)$-component $(Z_{\phi})_{ij}$ is $\phi(a_i,a_j)$. A linear system $\V f = Z_\phi \V g$ expresses each $f(a_i)$ as $\sum_{a_j:a_i\to a_j} \phi(a_i,a_j)g(a_j)$. In certain cases the matrix $Z_\phi$ can be inverted to express $\vec g$ as a function of $\vec f$. For instance, when $\cat A$ is a poset and $\phi=\zeta$, the matrix $Z=Z_\zeta$ is invertible and its inverse is known as the M\"obius function; it was introduced by Rota in \cite{Rota}, as a generalization of the number-theoretic M\"obius function. Similarly, if $\phi$ is a probabilistic transition kernel and $Z_\phi$ is invertible, this  process might be seen as an inversion of a system of conditional expectations.
\end{remark}

We define a category $\cat{ProbFinCat}$ of probabilistic (finite) categories whose objects are categorical probabilistic triples. A morphism $F: (\cat{A}, p, \phi) \to (\cat{B}, q, \theta)$ in $\cat{ProbFinCat}$ is given by a functor $F: \cat{A} \to \cat{B}$ 
such that for all $b \in \Ob(\cat{B})$,
\begin{equation}\label{eq:push-forward-probas}
q(b) = F_{*}p(b) = \sum_{a\in F^{-1}(b)} p(a).
\end{equation}
and for all $b,b' \in \Ob(\cat{B})$,
\begin{equation}\label{eq:image_kernel}
\theta(b,b') = F_{*}\phi(b,b') = 
\begin{cases}
\frac{\sum_{a'\in F^{-1}(b')}p(a')\sum_{a\in F^{-1}(b)}\phi(a,a')}{F_{*}p(b')} & F_{*}p(b')>0\\
1 & b=b', F_{*}p(b') = 0\\
0 & b\neq b', F_{*}p(b') = 0
\end{cases}.
\end{equation}
Remark that \eqref{eq:push-forward-probas} corresponds to the push-forward of probabilities under the function induced by $F$ on objects. In turn, when $b\mapsto \phi(b,a)$ is a probability of transition,  \eqref{eq:image_kernel} 
gives a probability of a transition $\theta(b,b')$ from $b'$ to $b$ in $\cat B$ as a weighted average of all the transitions from preimages of $b'$ to preimages of $b$. 

Lemma \ref{lem:well-defined-functor}  shows that the function $F_*\phi$ defined by \eqref{eq:image_kernel} is compatible with our definition of a categorical probabilistic triple. Lemma \ref{lem:functoriality} establishes the functoriality of \eqref{eq:push-forward-probas} and \eqref{eq:image_kernel}. 

\begin{lemma}\label{lem:well-defined-functor}
    Let $(\cat{A}, p, \phi)$ be a categorical probabilistic triple, $\cat B$ a finite category, and. $F:\cat A\to \cat B$ be a functor. Then $F_*\phi(b,b)>0$ for all $b\in \Ob(\cat B)$, and $F_*\phi(b,b')=0$ whenever $\Hom(b,b')=\emptyset$.
\end{lemma}
\begin{proof}
    Let $b, b'$ be  objects of $\cat B$ and suppose that $F_*p(b)>0$ (otherwise $F_*\phi(b,b)>0$ and $F_*\phi(b,b')=0$ by definition). 
    
    To prove the first claim, remark that
    \begin{equation}
        F_*\phi(b,b) \geq \sum_{a\in F^{-1}(b)} p(a) \phi(a,a) \geq \min_{a\in F^{-1}(b)} \phi(a,a) F_*p(b) > 0.
    \end{equation}

If $\Hom(b,b')=\emptyset$ then $\Hom(a,a')=\emptyset$ for any $a\in F^{-1}(b)$ and $a'\in F^{-1}(b')$; it follows that $\phi(a,a')=0$ by Definition \ref{def:cat_prob_triples}. Then it is clear from \eqref{eq:image_kernel} that $F_*\theta(b,b')$ vanishes.
\end{proof}

\begin{lemma}\label{lem:functoriality}
    Let $(\cat A,p,\phi)\overset{F}{\to} (\cat B, q,\theta)\overset{G}{\to} (\cat C, r,\psi)$ be a diagram in $\cat{ProbFinCat}$. Then $(G\circ F)_*p=G_*(F_*p)=G_*q$ and $(G\circ F)_* \phi = G_*(F_*\phi) = G_*\theta$. 
\end{lemma}
\begin{proof}
    For any $c\in \cat C$,
    \begin{equation*}
        (G\circ F)_*p(c) = \sum_{a\in (G\circ F)^{-1}(c)} p(a) = \sum_{b\in G^{-1}(c)}\sum_{a\in F^{-1}(b)} p(a) = \sum_{b\in G^{-1}(c)} F_*(b).
    \end{equation*}

    Similarly, for any $c,c'\in \Ob (\cat C)$, 
    \begin{align*}
        (G\circ F)_*  &\phi(c,c') = \frac{1}{G_*(F_*p)(c)} \sum_{b'\in G^{-1}(c')}\sum_{a'\in F^{-1}(b')} p(a') \sum_{b\in G^{-1}(c)}\sum_{a\in F^{-1}(b)} \phi(a,a') \nonumber\\
        &= \frac{1}{G_*q(c)} \sum_{b'\in G^{-1}(c)} q(b') \sum_{b\in G^{-1}(c)} \left(\sum_{a'\in F^{-1}(b')} \frac{p(a')}{F_*p(b')} \sum_{a\in F^{-1}(b)} \phi(a,a')\right) \nonumber \\
        &=G_*(F_*\phi)(c,c') =G_*\theta(c,c').
    \end{align*}
\end{proof}

We also consider probability preserving products and weighted sums as follows. For any two categorical probabilistic triples $(\cat{A}, p, \phi), (\cat{B}, q, \theta)$, we define their probability preserving product $(\cat{A}, p, \phi) \otimes (\cat{B}, q, \theta)$ to be the triple $(\cat{A} \times \cat{B}, p\otimes q, \phi\otimes \theta)$ where for any $\langle a,b\rangle, \langle a', b'\rangle\in \Ob(\cat{A\times \cat B}),$
\begin{equation}
(p\otimes q)(\langle a,b \rangle) = p(a)q(b), \quad \text{and}\quad (\theta\otimes \phi)(\langle a,b \rangle, \langle a', b' \rangle) = \theta(a,a')\phi(b,b').
\end{equation}

Given any $\lambda \in [0,1]$, we define the weighted sum $(\cat{A}, p, \phi) \oplus_{\lambda} (\cat{B}, q, \theta)$ by $(\cat{A} \coprod \cat{B}, p\oplus_{\lambda}q, \phi\oplus \theta)$ where
\begin{equation}
(p\oplus_{\lambda}q)(x) = 
\begin{cases}
\lambda p(x) & x\in \Ob(\cat{A})\\
(1-\lambda) q(x) & x \in \Ob(\cat{B})
\end{cases}
\end{equation}
and
\begin{equation}
(\phi\oplus \theta)(x,y) = 
\begin{cases}
\phi(x,y) & x,y \in \Ob(\cat{A})\\
\theta(x,y) & x,y \in \Ob(\cat{B})\\
0 & \text{otherwise}
\end{cases}.
\end{equation}
Given morphisms $f_i:(\cat A_i,p_i\phi_i)\to (\cat B_i,q_i,\theta_i)$, for $i=1,2$, there is a unique morphism
\begin{equation}
    \lambda f_1\oplus (1-\lambda)f_2:(\cat A_1,p_1\phi_1)\oplus_{\lambda} (\cat A_2,p_2\phi_2) \to (\cat B_1,q_1,\theta_1) \oplus_{\lambda} (\cat B_2,q_2,\theta_2)
\end{equation}
that restricts to $f_1$ on $\cat A_1$ and to $f_2$ on $\cat A_2$. 

 Finally, we introduce a notion of continuity for functions defined on categorical probabilistic triples. Let $\{(\cat{A}_k, p_k, \phi_k)\}_{k\in \mathbb{N}}$ be a sequence of categorical probabilistic triples; we say that $\{(\cat{A}_k, p_k, \phi_k)\}_k$ converges to an object $(\cat{A}, p, \phi)$ of $\cat{ProbFinCat}$ if $\cat{A}_k = \cat{A}$ for sufficiently large $k$ and $\{p_k\}_k, \{\phi_k\}_k$ converge pointwise as sequences of functions to $p$ and $\phi$ respectively. A function $G: \Ob (\cat{ProbFinCat}) \to \mathbb{R}$ is continuous if for any convergent sequence $\{(A_k, p_k, \phi_k)\}_{k\in \mathbb{N}} \to (A, p, \phi)$, the sequence $\{G(A_k, p_k, \phi_k)\}_{k \in \mathbb{N}}$ converges to $G(A, p, \phi)$. 

 Similarly, a sequence of morphisms $f_n:(\cat A_n, p_n,\theta_n)\to (\cat B_n,q_n,\psi_n)$ coverges to $f:(\cat A, p,\theta)\to (\cat B,q,\psi)$ if $\cat A_n=\cat A$, $\cat B_n = \cat B$ and $f_n = f$ (as functors) for $n$ big enough, and $p_n\to p$ and $\theta_n\to \theta$ converge pointwise. Functors from $\cat{ProbFinCat}$ to a topological (semi)group|seen as a one-point category|are continuous if they are sequentially continuous. One obtains in this way a generalization of the notions of continuity discussed in \cite[p. 4]{Entropy_Characterization} to our setting, which is compatible with the faithful embedding of the category $\cat{FinProb}$ of finite probability spaces considered there into $\cat{ProbFinCat}$ which maps a finite set equipped with a probability $(A,p)$ to the triple $(\cat A, p,\delta)$, where $\cat A$  is the discrete category with object set $A$ and $\delta$ is the Kronecker delta function, given by $\delta(x,x)=1$ for all $x\in \Ob \cat A$, and $\delta(x,y)=0$ whenever $x\neq y$.

\section{Categorical entropy}\label{sec:entropy}
\label{Sec:generalized_entropy}


Let $Z$ be a general matrix in $\text{Mat}_{N\times N}([0,\infty))$ with strictly positive diagonal entries and $p$ a probability on a finite set of cardinality $N$. In the context of \cite[Ch.~6]{LeinsterDiversityBook}, which discusses measures of ecological diversity, $Z$ corresponds to a species similarity matrix and $p$ to the relative abundance of a population of $N$ different species. 
The species diversity of order $1$ is given by 
\begin{equation}
^1D^Z(p) = \prod_{i=1}^N \frac{1}{(Zp)_i^{p_i}}
\end{equation}
and its logarithm is a generalization of Shannon entropy, which is recovered when $Z$ is the identity matrix. 

Inspired by this, we consider the following entropic functional:
\begin{equation}\label{eq:def_cat_entropy}
    \mathcal H(\cat A, p,\phi) := -\sum_{a\in \Ob \cat A} p(a) \ln \left(\sum_{b \in \Ob \cat A} p(b)\phi(a,b)\right) =-\sum_{i=1}^N p_i \ln ((Z_\phi p)_i).
\end{equation}
As we required $\phi(a,a)>0$ for any $a\in \Ob (\cat{A})$, we have that $\sum_{b \in \Ob \cat A} p(b)\phi(a,b)\geq p(a)\phi(a,a) > 0$ whenever $p(a)\neq0$. In order to preserve continuity (see Proposition \ref{prop:continuity} below) while making $\mathcal{H}$ well defined, we take the convention that
\begin{equation*}
p(a)\ln\left(\sum_{b\in\Ob(\cat{A})} p(b)\phi(a,b)\right)=0
\end{equation*}
whenever $p(a)=0$. 

The rest of this section establishes certain properties of the functional $\mathcal H$. These properties only depend on the operations $\oplus$ and $\otimes$, as well as the topology on the resulting semiring (a.k.a. \emph{rig}) of categorical probabilistic triples. The morphisms in $\cat{ProbFinCat}$ only appear in the next section, in connection with an algebraic characterization of $\mathcal H$. 

The functional $\mathcal H$ generalizes Shannon entropy in the following sense.

\begin{proposition}
\label{prop:identity}
    If $\phi=\delta$, the Kronecker delta, then $\mathcal H(\cat A,p, \phi) = H(p)=-\sum p_i\ln p_i$.
\end{proposition}

We now consider properties of $\mathcal{H}$ that are analogous to those used to characterize Shannon entropy, see \cite{Csiszar2008}.

\begin{proposition}
    $\mathcal H((\cat A,p,\phi)\otimes (\cat B,q,\theta)) = \mathcal H(\cat A,p,\phi) + \mathcal H(\cat B,q,\theta).$
\end{proposition}
\begin{proof}
    For simplicity, we write $\cat A_0 := \Ob(\cat A)$, etc. Recall that $(\cat A \times \cat B)_0= \cat A_0\times \cat B_0$. By a direct computation,
    \begin{multline*}
        \mathcal H((\cat A,p,\phi)\otimes (\cat B,q,\theta)) \\ = -\sum_{(a,b)  \in \cat A_0 \times \cat B_0} p(a)q(b) \ln \left(\sum_{(a',b') \in  \cat A_0 \times \cat B_0} p(a')q(b')\phi(a,a')\theta(b,b')\right)\\
        = -\sum_{(a,b)  \in \cat A_0 \times \cat B_0} p(a)q(b) \ln \left(\left(\sum_{a' \in \cat A_0 } p(a')\phi(a,a')\right)\left(\sum_{b' \in \cat B_0 } q(b')\theta(b,b')\right)\right),
    \end{multline*}
    from which the result follows. More concisely, one might remark that  $Z_{\phi\otimes \theta}$ equals $ Z_\phi \otimes Z_\theta$, the Kronecker product of matrices, and that $(Z_\phi\otimes Z_\theta)(p\otimes q) = Z_\phi p \otimes Z_\theta q$. 
\end{proof}

\begin{proposition}
\label{prop: weighted_coproduct}
Given any $\lambda \in [0,1]$, let $(\cat{2}, \Lambda, \delta)$ be the categorical probability triple where $\cat{2}$ denotes the discrete category with exactly two objects, $\Ob(\cat{2}) = \{x_1,x_2\}$, $\Lambda$ is such that $\Lambda(x_1) = \lambda, \Lambda(x_2) = 1-\lambda$, and  $\delta$ denotes the Kronecker delta. Then, 
\begin{equation*}
\mathcal{H}((\cat A,p,\phi)\oplus_{\lambda} (\cat B,q,\theta)) = \lambda \mathcal{H}(\cat A,p,\phi) + (1-\lambda)\mathcal{H}(\cat{B}, q, \theta) + \mathcal{H}(\cat{2},\Lambda, \delta).
\end{equation*}
\end{proposition}
\begin{proof}
By a direct computation: $\mathcal{H}((\cat A,p,\phi)\oplus_{\lambda} (\cat B,q,\theta))$ equals
\begin{multline*}
 -\sum_{a\in \Ob(\cat{A})}\lambda p(a) \ln\left(\sum_{b\in \Ob(\cat{A})}\lambda p(b)\phi(a,b)  \right)
\\ \qquad - \sum_{x\in \Ob(\cat{B})}(1-\lambda)q(x) \ln \left(\sum_{y\in \Ob(\cat{B})} (1-\lambda)q(y)\theta(x,y)\right) 
\end{multline*}
from which the result follows, because $\mathcal H(\cat 2,\Lambda,\delta)=- \lambda\ln(\lambda)-(1-\lambda)\ln(1-\lambda)$. Similar to the product case, one might remark that $Z_{\phi\oplus \theta} = Z_\phi \oplus Z_\theta$ and $(Z_\phi \oplus Z_\theta)(p\oplus_{\lambda}q) = Z_\phi(\lambda p) \oplus Z_\theta((1-\lambda)q)$. 
\end{proof}

More generally, given a finite collection $(\cat{A}_1,p_1,\phi_1),\ldots, (\cat{A}_m, p_m, \phi_m)$   of categorical probabilistic triples and a probability vector $(\lambda_1,...,\lambda_m)$, we can define 
\begin{equation*}
\bigoplus_{i=1}^m \lambda_i (\cat{A}_i, p_i, \phi_i) = \left(\coprod_{i=1}^m \cat{A}_i, \bigoplus_{i=1}^m\lambda_i p_i, \bigoplus_{i=1}^m \phi_i\right)
\end{equation*}
where $
(\bigoplus_{i=1}^m\lambda_i p_i)(x) = 
\lambda_ip_i(x)$ for $ x \in \Ob(\cat{A}_i)$,  and 
\begin{equation}
    \left(\bigoplus_{i=1}^m \phi_i\right)(x,y) = \begin{cases}
    \phi_i(x,y) & \text{if }x,y \in \Ob(\cat{A}_i)\\
    0& \text{otherwise}
\end{cases} .
\end{equation}

\begin{proposition}
\label{prop:family_weighted_coprod}
Take notation as in the preceding paragraph. Let $(\cat{m}, \Lambda, \delta)$ be the categorical probability triple where $\cat{m}$ is the finite discrete category of cardinality $m$, with objects $\{x_1,\ldots, x_m\}$,  $\Lambda(x_i) = \lambda_i$, and $\delta$ is again the Kronecker delta. Then,
\begin{equation*}
\mathcal{H}\left(\bigoplus_{i=1}^m \lambda_i (\cat{A}_i, p_i, \phi_i)\right)
= \sum_{i=1}^m \lambda_i \mathcal{H}(\cat{A}_i,p_i, \phi_i) + \mathcal{H}(\cat{m}, \Lambda, \delta).
\end{equation*}
\end{proposition}
\begin{proof}
Similar to the proof of Proposition \ref{prop: weighted_coproduct}, observe that $Z_{\bigoplus_{i=1}^m \phi_i} = \bigoplus_{i=1}^m Z_{\phi_i}$ and $(\bigoplus_{i=1}^m Z_{\phi_i})(\bigoplus_{i=1}^m \lambda_i p_i) = \bigoplus_{i=1}^m Z_{\phi_i}(\lambda_i p_i)$. 
\end{proof}

\begin{proposition}
\label{prop:continuity}
The entropy functional $\mathcal{H}$ is continuous. 
\end{proposition}
\begin{proof}
Let $\{(\cat{A}_k, p_k, \phi_k)\}_{k\in\mathbb{N}}$ be a sequence of categorical probabilistic triples converging to the triple $(\cat{A}_\infty, p_\infty, \phi_\infty)$. It follows that, for any $a\in \Ob \cat A$, the sequence $f_k(a):=\sum_{b\in \Ob (\cat A)} p_k(b) \phi_k(a,b)$ converges pointwise to $f_\infty$ as $k\to \infty$. Since we may rewrite
\begin{equation*}
\mathcal{H}(\cat{A}_k, p_k, \phi_k) = -\sum_{a\in \Ob(\cat{A})} p_k(a)\ln(f_k(a))
\end{equation*}
for $k=0, 1, 2, 3,...,\infty$, it suffices to show that $\{p_k(a)\ln(f_k(a))\}_{k\in\mathbb{N}}$ converges to $p_\infty(a)\ln(f_\infty(a))$ for each $a \in \Ob(\cat{A})$. 

Fix $a \in \Ob(\cat{A})$. Assume first that $p_\infty(a)>0$. Then for sufficiently large $k$, $p_k(a)>0$, and hence $f_k(a) \geq p_k(a)\phi_k(a,a) > 0$ and $f_\infty(a) \geq p_\infty(a)\phi_\infty(a,a) > 0$. By the continuity of $\ln(x)$, we then have that $\ln(f_k(a))$ converges to $\ln(f_\infty(a))$, hence $
\lim_{k \to \infty} p_k(a)\ln(f_k(a)) = p_\infty(a)\ln(f_\infty(a)). $

Assume now that $p_\infty(a) = 0$. If $f_\infty(a) > 0$, then $f_k(a)>0$ for sufficiently large $k$, so we may again use the continuity of $\ln(x)$. 
If $f_\infty(a) = 0$, pick $1 > \epsilon, \epsilon_0 > 0$ so that there exists $N \in \mathbb{N}$ such that $f_k(a) < \epsilon$ and $\phi_k(a,a) > \epsilon_0$ for all $k > N$. Then,
\begin{equation*}
1 > \epsilon > f_k(a) \geq p_k(a)\phi_k(a,a) \geq p_k(a)\epsilon_0.
\end{equation*}
We deduce from this that
\begin{equation*}
0 \geq p_k(a)\ln(f_k(a)) \geq p_k(a)\ln(p_k(a)) + p_k(a)\ln(\epsilon_0). 
\end{equation*}
 Using that $\lim_{x\to 0^+} x\ln(x) = 0$, we conclude that 
$\lim_{k \to \infty} p_k(a)\ln(f_k(a))= 0 = p_\infty(a)\ln(f_\infty(a)).$
We thus have that $\mathcal{H}$ is continuous. 
\end{proof}

We take a brief detour to recall a definition of magnitude for a matrix. 
\begin{definition}
\label{def: magnitude_matrix}
Let $M$ be a matrix in $\text{Mat}_{n\times n}(\mathbb{R})$. Denote by $1_n \in \text{Mat}_{n\times 1}(\mathbb{R})$ the column vector of ones. 
A matricial weighting of $M$ is a (column) vector $w \in \text{Mat}_{n\times 1}(\mathbb{R})$ such that $Mw=1_n$
Similarly, a matricial coweighting of $M$ is a  vector $\bar w \in \text{Mat}_{n\times 1}(\mathbb{R})$ such that $\bar w^T M = 1_n$.

 If $M$ has both a weighting $w$ and a coweighting  $\bar w$, then we say that $M$ has magnitude, with the magnitude of $M$ given by
\begin{equation}\label{eq:magnitude_matrix}
\norm{M}:= \sum_{i=1}^n w_i = \sum_{i=1}^n \bar w_i.
\end{equation}
\end{definition}

In fact, it follows easily from the definitions that if $M$ has both a weighting and a coweighting the sums in \eqref{eq:magnitude_matrix} are equal.

\begin{proposition}
\label{prop:general_mag}
    Let $(\cat A, p,\phi)$ be a probabilistic category. Suppose $Z_\phi$ has magnitude (i.e. has a weighting and a coweighting). If $Z_\phi$ has a nonnegative weighting $w$, then $u = w/\norm{Z_\phi}$ is a probability distribution and $
        \mathcal H(\cat A,u, \phi) = \ln \norm{Z_\phi}.$
\end{proposition}
\begin{proof}
Remark that if the weighting is nonnegative then necessarily $\norm{Z_\phi}>0$. Additionally, 
\begin{equation}
    \mathcal H(\cat A,u, \phi)= -\sum_i \frac{w_i}{\norm{Z_\phi}} \ln\left(\frac 1{\norm{Z_\phi}}\right).
\end{equation}
\end{proof}

Specialized to the case $\phi=\zeta$, Proposition \ref{prop:general_mag} tells us that if $\cat A$ has magnitude, and the category (equivalently: the matrix $Z$ representing $\zeta$) has a positive weighting $w$, then $u=w/\chi(\cat A)$ satisfies $\mathcal H(\cat A, u ,\zeta) = \ln \chi(\cat A)$. Therefore the categorical entropy generalizes both Shannon entropy and the logarithm of the categorical magnitude, as we wanted.

\begin{remark}
    A known result on the maximization of diversity \cite[Thm. 2]{leinster2016maximizing} can be restated as follows: the supremum of  $\mathcal H(\cat A, p,\phi)$ over all probability distributions $p$ on $\Ob (\cat A)$ equals $\ln(\max_{B} \norm{Z_B})$,
    where the maximum is taken over all subsets $B$ of $\{1,...,n\}$ such that the submatrix $Z_B:=((Z_\phi)_{i,j})_{i,j\in B}$ of $Z_\phi$ has a nonnegative weighting.
\end{remark}

\begin{remark}
    One might generalize the definitions, allowing $p$ to be a \emph{signed probability}: a function $p:\Ob(\cat{A})\to \mathbb R$ such that $\sum_{a\in \Ob(\cat{A})} p(a)=1$. In this case, we define $\mathcal H(\cat A, p,\phi)$ as $-\sum_{i=1}^N p_i \ln |(Z_\phi p)_i|.$ Then the previous proposition generalizes as follows: if $\norm {Z_\phi}\neq 0$ and $w$ is any weighting, the vector $u=w/\norm{Z_\phi}$ is a signed probability and $\mathcal H(\cat A,u, \phi) = \ln | \norm{Z_\phi}|$.
\end{remark}

\section{Towards a characterization of the categorical entropy}\label{sec:special_cases}

Let $\cat R_+$ be the additive semigroup of non-negative real numbers seen as a one-object category, that is, $\Ob (\cat R_+) =\{\ast\}$  and $\Hom(\ast,\ast)=[0,\infty)$, with $+$ as composition of arrows. A functor $F:\cat{ProbFinCat}\to \cat R_+$ is:
\begin{itemize}
    \item \textbf{convex-linear} if for all morphisms $f$, $g$ and all scalars $\lambda\in[0,1]$, $$F(\lambda f \oplus (1-\lambda) q) =\lambda F(f) + (1-\lambda) F(g).$$
    \item \textbf{continuous} if $F(f_n)\to F(f)$ whenever $f_n$ is a sequence of morphisms converging to $f$ (see Section \ref{sec:prob_cat}).
\end{itemize}

Via the embedding described at the end of Section \ref{sec:prob_cat}, we obtain similar definitions for functors $F:\cat{FinProb}\to \cat R_+$, which correspond to those introduced by Baez, Fritz, and Leinster in \cite[p. 4]{Entropy_Characterization}. They showed there that if a functor $F:\cat{FinProb}\to\cat R_+$ is convex linear and continuous, then there exists a constant $c\geq 0$ such that  each arrow $f:(A,p)\to (B,q)$ is mapped to  $F(f) = c(H(p)-H(q))$. This is an algebraic characterization of the \emph{information loss} $H(p)-H(q)$ given by Shannon entropy.

In our setting, the ``loss'' functor $L:\cat{ProbFinCat}\to \cat{R_+}$ that maps $f:(\cat A,p,\theta)\to (\cat B,q,\psi)$ to $L(f) = \mathcal H(\cat A,p,\theta)- \mathcal H(\cat B,q,\psi)$ is continuous and convex linear: this can be easily seen from Propositions \ref{prop:family_weighted_coprod} and \ref{prop:continuity} above. However, are all the convex-linear continuous functors $G:\cat{ProbFinCat}\to \cat{R_+}$ positive multiples of $L$? 

Baez, Fritz, and Leinster's proof  uses crucially that $\cat{FinProb}$ has a terminal object. $\cat{ProbFinCat}$ has no terminal object: even a category with a unique object $\ast$ accepts an arbitrary value of $\theta(\ast,\ast)$. To fix this, we introduce the full subcategory $\cat{TransFinCat}$ of $\cat{ProbFinCat}$ given by triples $(\cat{A}, p, \phi)$ such that $\phi$ is a transition kernel; this means that for all $a\in \Ob (\cat A)$, the function $b\mapsto \phi(b,a)$ is a probability measure. It is easy to verify that if $F: (\cat{A}, p, \phi) \to (\cat{B}, q, \theta)$ is a morphism in $\cat{ProbFinCat}$ and $\phi$ is a transition kernel, then $\theta$ is a transition kernel too. 


Note that the embedding of $\cat{FinProb}$ into $\cat{ProbFinCat}$ described previously induces a fully faithful embedding of $\cat{FinProb}$ into $\cat{TransFinCat}$; moreover, when $\cat{A}$ is a discrete category, the only choice of transition kernel is given by the Kronecker delta $\delta$.

If $p$ represents 
the probability of each object at an initial time, we can regard $\hat p:\Ob (\cat A)\to [0,1], \, a\mapsto \sum_{b\in \Ob (\cat A)} \phi(a,b)p(b)$ as the probability of each object after one transition has happened. Quite remarkably, there is a compatibility between the push-forward of probabilities \eqref{eq:push-forward-probas} and the push-forward of transition kernels \eqref{eq:image_kernel} in the following sense: if $F:(\cat A, p, \phi)\to (\cat B, q, \theta)$ is a morphism in $\cat{TransFinCat}$, then $\hat q = F_* \hat p$, because
    \begin{align}
    \hat q(b) & := \sum_{b'\in \Ob (\cat B)} \theta(b,b') q(b')\\
    &= \sum_{b'\in \Ob(\cat B)} \left(\frac{\sum_{a'\in F^{-1}(b') }p(a')\sum_{a\in F^{-1}(b) } \phi(a,a')}{q(b')}\right) q(b')\\
    &=\sum_{a\in F^{-1}(b)} \sum_{a'\in \Ob(\cat A)} \phi(a,a') p(a')=\sum_{a\in F^{-1}(b)} \hat p(a).
\end{align}
In turn, the entropy \eqref{eq:def_cat_entropy} can be rewritten as
\begin{equation}
    \mathcal H(\cat A, p,\phi) = -\sum_{a\in \Ob\cat A} p(a) \ln \hat p(a) = H(p)-D(\hat p|p),
\end{equation}
where  $H(p)$ is Shannon entropy \eqref{eq:entropy}, and $D(\hat p|p) = \sum_{a\in \Ob \cat A} p(a)\ln(\hat p(a)/p(a))$ is the Kullback-Leibler divergence

Let us return to the characterization problem described above. Let $G:\cat{TransFinCat}\to \cat R_+$ be a functor that is convex linear, continuous, and possibly subject to other constraints. Let $\top$ be the triple (one-point discrete category, trivial probability, trivial transition kernel), which is the terminal object of $\cat{TransFinCat}$. Denote by $I_G(\cat A, p,\phi)$ the image under $G$ of the unique morphism $(\cat A, p,\phi)\to \top$. The functoriality of $G$ implies that, for any morphism $f:(\cat A, p,\phi)\to (\cat B, q,\theta)$, one has $G(f) + I_G(\cat B, q,\theta) = I_G(\cat A, p,\phi)$. In particular, $G(\id_{(\cat A,p,\phi)})=0$ for any triple $(\cat A,p,\phi)$ and $I_G(\top)=0$. Moreover, if $f$ is invertible, then $G(f)+G(f^{-1})=0$, so $G(f)=0$ and $I_G$ is invariant under isomorphisms. In turn, convex linearity implies that 
\begin{equation}\label{eq:chain_rule}
    I_G\left(\bigoplus_{i=1}^m \lambda_i (\cat A_i, p_i,\phi_i)\right)= I_G(\cat m, (\lambda_1,...,\lambda_m),\delta) + \sum_{i=1}^m \lambda_i I_G (\cat A_i, p_i,\phi_i),
\end{equation}
for any vector of probabilities $(\lambda_1,...,\lambda_m)$ and any $(\cat A_i, p_i,\phi_i)\in \Ob (\cat{TransFinCat})$, $i=1,..,m$. This is a system of functional equations reminiscent of those used by Shannon, Khinchin, Fadeev, etc. to characterize the entropy \cite{Shannon,Csiszar2008}, see in particular \cite[Thm. 6]{Entropy_Characterization}. It is not clear if a similar result holds for $\mathcal H$ in our categorical setting. A fundamental difference is that every finite probability space $(A,p)$ can be expressed nontrivially in many ways as convex combinations $ \bigoplus_{i=1}^m \lambda_i ( A_i, p_i)$ in $\cat{FinProb}$, but this is not true for categorical probabilistic triples. Remark also that the function $I_G(\cat A,p,\phi) = -\sum_{a\in \Ob (\cat A)} p(a)\ln(\sum_{b\in \cat{A}} \phi(b,a)p(b))$ defines a continuous, convex-linear functor $G$.  Hence $G$  needs to satisfy some additional properties in order to recover $L$ up to a positive multiple.


\bibliographystyle{unsrt}
\bibliography{entropy_of_categories}

\end{document}